\documentclass[10pt]{amsart}

\usepackage{amssymb,amsthm}
\usepackage{graphicx}
\usepackage{mathrsfs}
\usepackage{amsmath}
\usepackage{verbatim}

\newtheorem{theorem}{Theorem}
\newtheorem{proposition}{Proposition}
\newtheorem{lemma}{Lemma}

\newtheorem{corollary}{Corollary}

\theoremstyle{definition}

\def\xxi{{\boldsymbol{\xi}}}
\def\zzeta{{\boldsymbol{\zeta}}}
\def\f{\bar f}
\def\pp{\scrP_\Lambda'}
\def\bM{\mathbf M}
\def\bY{\mathbf Y}
\def\bz{\mathbf z}

\def\RR{{\mathbb R}}
\def\calA{\mathcal A}
\def\scrP{\mathscr P}
\def\calF{\mathcal F}
\def\calX{\mathcal X}
\def\bH{\boldsymbol H}
\def\bh{\boldsymbol h}
\def\T{\top}
\def \1{{\rm 1}\mskip -4,5mu{\rm l} }
\def\l{\lambda}
\def\tr{{\rm Tr}}
\def\KL{\mathcal K}
\def\cal#1{\mathcal#1}

\begin{document}

\title[Aggregation and sparsity]{Aggregation by exponential
weighting, sharp PAC-Bayesian bounds and sparsity}

\author{A.\ Dalalyan \and A.B.\ Tsybakov$^{\dag}$}

\address{A.\ Dalalyan: LPMA, University of Paris 6\\
4, Place Jussieu, 75252 Paris cedex 05, France}

\address{A.B.\
Tsybakov: Laboratoire de
Statistique, CREST, Timbre J340 \\
3, av.\ Pierre Larousse, 92240 Malakoff cedex, France \\
and LPMA, University of Paris 6\\
4, Place Jussieu, 75252 Paris cedex 05, France}

\thanks{$^\dag$Research partially supported by the grant ANR-06-BLAN-0194 and by the PASCAL
Network of Excellence}

\begin{abstract}
We study the problem of aggregation under the squared loss in the
model of regression with deterministic design. We obtain sharp
PAC-Bayesian risk bounds for aggregates defined via exponential
weights, under general assumptions on the distribution of errors and
on the functions to aggregate. We then apply these results
to derive sparsity oracle inequalities.
\end{abstract}



\maketitle

\section{Introduction}

Aggregation with exponential weights is an important tool in machine
learning. It is used for estimation, prediction with expert advice,
in PAC-Bayesian settings and other problems. In this paper we
establish a link between aggregation with exponential weights and
sparsity. More specifically, we obtain a new type of oracle
inequalities and apply them to show that the exponential weighted
aggregate with a suitably chosen prior has a sparsity property.

We consider the regression model
\begin{equation}
\label{0} Y_i= f(x_i) + \xi_i, \quad i=1,\dots, n,
\end{equation}
where $x_1,\dots,x_n$ are given non-random elements of a set ${\cal
X}$, $f:{\cal X} \to \RR$ is an unknown function, and $\xi_i$ are
i.i.d. zero-mean random variables on a probability space
$(\Omega,\calF,P)$ where $\Omega\subseteq \RR$. The problem is to
estimate the function $f$ from the data
$D_n=((x_1,Y_1),\dots,(x_n,Y_n))$.

Let $(\Lambda,\calA)$ be a measurable space and denote by
$\scrP_\Lambda$ the set of all probability measures defined on
$(\Lambda,\calA)$. Assume that we are given a family
$\{f_\lambda,\,\lambda\in\Lambda\}$ of functions $f_\lambda:\mathcal
X\to\RR$ such that the mapping $\lambda\mapsto f_\lambda(x)$ is
measurable for all $x\in \mathcal X$, where $\RR$ is equipped with
the Borel $\sigma$-field. Functions $f_\lambda$ can be viewed either
as weak learners or as some preliminary estimators of $f$ based on a
training sample independent of $\bY\triangleq(Y_1,\ldots,Y_n)$ and
considered as frozen.

We study the problem of aggregation of functions in
$\{f_\lambda,\,\lambda\in\Lambda\}$ under the squared loss. The aim
of aggregation is to construct an estimator $\hat f_n$ based on the
data $D_n$ and called the {\it aggregate} such that the expected
value of its squared error
$$
\|\hat f_n-f\|_n^2\triangleq\frac1n\sum_{i=1}^n\big(\hat
f_n(x_i)-f(x_i)\big)^2
$$
is approximately as small as the oracle value
$\inf_{\lambda\in\Lambda} \|f-f_\lambda\|_n^2$.

In this paper we consider aggregates that are mixtures of functions
$f_\l$ with exponential weights. For a measure  $\pi$ from
$\scrP_\Lambda$ and for $\beta>0$ we set
\begin{equation}\label{aggr}
\hat f_n(x)\triangleq\int_\Lambda
\theta_\lambda(\bY)f_\lambda(x)\,\pi(d\lambda), \quad x\in \calX,
\end{equation}
with
\begin{equation}\label{ag1}
\theta_\lambda(\bY)=\frac{\exp\big\{-n\|\bY-f_\lambda\|_n^2/\beta\big\}}
{\int_\Lambda \exp\big\{-n\|\bY-f_w\|_n^2/\beta\big\}\pi(dw)}\,
\end{equation}
where $\|\bY-f_\lambda\|_n^2\triangleq
\frac1n\sum_{i=1}^n\big(Y_i-f_\lambda(x_i)\big)^2$ and we assume
that $\pi$ is such that the integral in (\ref{aggr}) is finite.

 Note that $\theta_\lambda(\bY)= \theta_\lambda(\beta,\pi,\bY)$, so that $\hat f_n$ depends on two
tuning parameters: the probability measure $\pi$ and the
``temperature" parameter $\beta$. They have to be selected in a
suitable way.

Using the Bayesian terminology, $\pi(\cdot)$ is a prior distribution
and $\hat f_n$ is the posterior mean of $f_\lambda$ in a ``phantom''
model
\begin{equation}\label{phm}
Y_i=f_\lambda(x_i)+\xi_i'
\end{equation}
where $\xi_i'$ are i.i.d. normally distributed random variables with
mean $0$ and variance $\beta/2$.

The idea of mixing with exponential weights has been discussed by
many authors apparently since 1970-ies (see \cite{y01} for an
overview of the subject). Most of the work has been focused on the
important particular case where the set of estimators is finite,
i.e., w.l.o.g. $\Lambda=\{1,\dots,M\}$, and the distribution $\pi$
is uniform on $\Lambda$. Procedures of the type
(\ref{aggr})--(\ref{ag1}) with general sets $\Lambda$ and priors
$\pi$ came into consideration quite recently
\cite{cat99,catbook:01,v:01,buno05,zhang1,zhang2,a1,a2}, partly in
connection with the PAC-Bayesian approach. For finite $\Lambda$,
procedures (\ref{aggr})--(\ref{ag1}) were independently introduced
for prediction of deterministic individual sequences with expert
advice. Representative work and references can be found in
\cite{v:90,lw94,cetal, kw99, lcb:06}; in this framework the results
are proved for cumulative loss and no assumption is made on the
statistical nature of the data, whereas the observations $Y_i$ are
supposed to be uniformly bounded by a known constant.

We mention also related work on cumulative exponential weighting
methods: there the aggregate is defined as the average
$n^{-1}\sum_{k=1}^n{\hat f_k}$. For regression models with random
design, such procedures are introduced and analyzed in \cite{cat99},
\cite{catbook:01} and \cite{y00}. In particular, \cite{cat99} and
\cite{catbook:01} establish a sharp oracle inequality, i.e., an
inequality with leading constant~1. This result is further refined
in \cite{buno05} and \cite{jrt05}. In addition, \cite{jrt05} derives
sharp oracle inequalities not only for the squared loss but also for
general loss functions. However, these techniques are not helpful in
the framework that we consider here, because the averaging device is
not meaningfully adapted to models with non-identically distributed
observations.

For finite $\Lambda$, the aggregate $\hat f_n$ can be computed on-line. This, in particular,
motivated its use for on-line prediction.
Papers \cite{jrt05}, \cite{jntv05} point out that $\hat f_n$ and its
averaged version can be obtained as a special case of mirror descent
algorithms that were considered earlier in deterministic
minimization. Finally, \cite{ccg04,jrt05} establish some links
between the results for cumulative risks proved in the theory of
prediction of deterministic sequences and generalization error
bounds for the aggregates in the stochastic i.i.d. case.

In this paper we obtain sharp oracle inequalities for the aggregate
$\hat f_n$ under the squared loss, i.e., oracle inequalities with
leading constant 1 and optimal rate of the remainder term. Such an
inequality has been pioneered in \cite{lb06} in a
somewhat different setting. Namely, it is assumed in \cite{lb06}
that $\Lambda$ is a finite set, the errors $\xi_i$ are Gaussian and
$f_\lambda$ are estimators constructed from the same sample $D_n$
and satisfying some strong restrictions (essentially, these should
be the projection estimators). The result of \cite{lb06} makes use
of Stein's unbiased risk formula, and gives a very precise constant
in the remainder term of the inequality. Inspection
of the argument in \cite{lb06} shows that it can be also applied in
the following special case of our setting: $f_\lambda$ are arbitrary
fixed functions, $\Lambda$ is a finite set and the errors $\xi_i$
are Gaussian.

The general line of our argument is to establish some PAC-Bayesian
risk bounds (cf.\  (\ref{oracle6}), (\ref{oracle})) and then to derive
sharp oracle inequalities by making proper choices of the
probability measure $p$ involved in those bounds (cf.\ Sections
\ref{sec4}, \ref{sec7}).

The main technical effort is devoted to the proof of the
PAC-Bayesian bounds (Sections \ref{sec2}, \ref{sec3}, \ref{sec5}).
The results are valid for general $\Lambda$ and arbitrary functions
$f_\lambda$ satisfying some mild conditions. Furthermore, we treat
non-Gaussian errors $\xi_i$. For this purpose, we suggest three
different approaches to prove the PAC-Bayesian bounds. The first one
is based on integration by parts techniques that generalizes Stein's
unbiased risk formula (Section \ref{sec2}). It is close in the
spirit to \cite{lb06}. This approach leads to most accurate results
but it covers only a narrow class of distributions of the errors
$\xi_i$. In Section \ref{sec3} we introduce another techniques based
on dummy randomization which allows us to obtain sharp risk bounds
when the distributions of errors $\xi_i$ are $n$-divisible. Finally,
the third approach (Section \ref{sec5}) invokes the Skorokhod
embedding and covers the class of all symmetric error distributions
with finite moments of order larger than or equal to 2. Here the
price to pay for the generality of the distribution of errors is in
the rate of convergence that becomes slower if only smaller moments
are finite.

In Section \ref{sec7} we analyze our risk bounds in the important
special case where $f_\lambda$ is a linear combination of $M$ known
functions $\phi_1,\ldots,\phi_M$ with the vector of weights
$\lambda=(\lambda_1,\dots,\lambda_M)$: $f_\lambda=\sum_{j=1}^M
\lambda_j\phi_j.$ This setting is connected with the following three
problems.

1.{\it High-dimensional linear regression.} Assume that the
regression function has the form $ f=f_{\lambda^*}$ where
$\lambda^*\in \RR^M$ is an unknown vector, in other words we have a
linear regression model. During the last years a great deal of
attention has been focused on estimation in such a linear model
where the number of variables $M$ is much larger than the sample
size $n$. The idea is that the effective dimension of the model is
defined not by the number of potential parameters $M$ but by the
unknown number of non-zero components $M(\lambda^*)$ of vector
$\lambda^*$ that can be much smaller than $n$. In this situation
methods like Lasso, LARS or Dantzig selector are used \cite{ehjt04,
ct}. It is proved that if $M(\lambda^*)\ll n$ and if the dictionary
$\{\phi_1,\ldots,\phi_M\}$ satisfies certain conditions, then the
vector $\lambda^*$ and the function $f$ can be estimated with
reasonable accuracy \cite{gr04,btw07a,btw07b,ct,zh07,brt07}.
However, the conditions on the dictionary $\{\phi_1,\ldots,\phi_M\}$
required to get risk bounds for the Lasso and Dantzig selector are
quite restrictive. One of the consequences of our results in Section
\ref{sec7} is that a suitably defined aggregate with exponential
weights attains essentially the same and sometimes even better
behavior than the Lasso or Dantzig selector with no assumption on
the dictionary, except for the standard normalization.

2.{\it Adaptive nonparametric regression.} Assume that $f$ is a
smooth function, and $\{\phi_1,\ldots,\phi_M\}$ are the first $M$
functions from a basis in $L_2(\RR^d)$. If the basis is orthonormal,
it is well-known that adaptive estimators of $f$ can be constructed
in the form $\sum_{j=1}^{M} \hat \lambda_j\phi_j$ where $\hat
\lambda_j$ are appropriately chosen data-driven coefficients and $M$
is a suitably selected integer such that $M\le n$ (cf.,e.g.,
\cite{nem,t04}). Our aggregation procedure suggests a more general
way to treat adaptation covering the problems where the system
$\{\phi_j\}$ is not necessarily orthonormal, even not necessarily a
basis, and $M$ is not necessarily smaller than $n$. In particular,
the situation where $M\gg n$ arises if we want to deal with sparse
functions $f$ that have very few non-zero scalar products with
functions from the dictionary $\{\phi_j\}$, but these non-zero
coefficients can correspond to very high ``harmonics". The results
of Section 7 cover this case.

3. {\it Linear, convex or model selection type aggregation.} Assume
now that $\phi_1$,$\dots$,$\phi_M$ are either some preliminary
estimators of $f$ constructed from a training sample independent of
$(Y_1,\dots,Y_n)$ or some weak learners, and our aim is to construct
an aggregate which is approximately as good as the best among
$\phi_1,\dots,\phi_M$ or approximately as good as the best linear or
convex combination of $\phi_1,\dots,\phi_M$. In other words, we deal
with the problems of model selection (MS) type aggregation or
linear/convex aggregation respectively \cite{nem,tsy:03}. It is
shown in \cite{btw07a} that a BIC type aggregate achieves optimal
rates simultaneously for MS, linear and convex aggregation. This
result is deduced in \cite{btw07a} from a sparsity oracle inequality
(SOI), i.e., from an oracle inequality stated in terms of the number
$M(\lambda)$ of non-zero components of $\lambda$. For a discussion
of the concept of SOI we refer to \cite{tsy06}. Examples of SOI are
proved in \cite{k06,btw06,btw07a, btw07b,vdg06, brt07} for the
Lasso, BIC and Dantzig selector aggregates. Note that the SOI for
the Lasso and Dantzig selector are not as strong as those for the
BIC: they fail to guarantee optimal rates for MS, linear and convex
aggregation unless $\phi_1,\dots,\phi_M$ satisfy some very
restrictive conditions. On the other hand, the BIC aggregate is
computationally feasible only for very small dimensions $M$. So,
neither of these methods achieves both the computational efficiency
and the optimal theoretical performance.

In Section 7 we propose a new approach to sparse recovery that
realizes a compromise between the theoretical properties and the
computational efficiency. We first suggest a general technique of
deriving SOI from the PAC-Bayesian bounds, not necessarily for our
particular aggregate $\hat f_n$. We then show that the exponentially
weighted aggregate $\hat f_n$ with an appropriate prior measure
$\pi$ satisfies a sharp SOI, i.e., a SOI with leading constant 1.
Its theoretical performance is comparable with that of the BIC in
terms of sparsity oracle inequalities for the prediction risk. No
assumption on the dictionary $\phi_1,\dots,\phi_M$ is required,
except for the standard normalization. Even more, the result is
sharper than the best available SOI for the BIC-type aggregate
\cite{btw07a}, since the leading constant in the oracle inequality
of \cite{btw07a} is strictly greater than~1. At the same time,
similarly to the Lasso and Dantzig selector, our method is
computationally feasible for moderately large dimensions~$M$.

\section{Some notation}

In what follows we will often write for brevity $\theta_\lambda$
instead of $\theta_\lambda(\bY)$. For any vector
$\bz=(z_1,\dots,z_n)^\T\in \RR^n$ set
$$
\|\bz\|_n =\left(\frac1{n}\sum_{i=1}^n z_i^2\right)^{1/2}.
$$
Denote by $\pp$ the set of all measures $\mu\in\scrP_\Lambda$ such
that $\lambda\mapsto f_\lambda(x)$ is integrable w.r.t. $\mu$ for
$x\in\{x_1,\dots,x_n\}$. Clearly $\pp$ is a convex subset of
$\scrP_\Lambda$. For any measure $\mu\in\pp$ we define
$$
\f_\mu(x_i)=\int_\Lambda f_\lambda(x_i)\,\mu(d\lambda), \quad
i=1,\dots,n.
$$
We denote by $\theta\cdot\pi$ the probability measure
 $A\mapsto \int_A \theta_\lambda\,\pi(d\lambda)$ defined on $\calA$.
 With
the above notation, we can write $$\hat f_n=\f_{\theta\cdot\pi}.$$

\section{A PAC-Bayesian bound based on unbiased risk estimation}
\label{sec2}

In this section we prove our first PAC-Bayesian bound. An important
element of the proof is an extension of Stein's identity which uses
integration by parts. For this purpose we introduce the function
$$
m_\xi(x)=-E[\xi_1\1(\xi_1\le x)]=-\int_{-\infty}^x z\, dF_\xi(z)=
\int_{x}^\infty z\, dF_\xi(z),
$$
where $F_\xi(z)=P(\xi_1\le z)$ is the c.d.f.\ of $\xi$, $\1(\cdot)$
denotes the indicator function and the last equality follows from
the assumption $E(\xi_1)=0$ . Since $E|\xi_1|<\infty$ the function
$m_\xi$ is well defined, non negative and satisfies
$m_\xi(-\infty)=m_\xi(+\infty)=0$. Moreover, $m_\xi$ is increasing
on $(-\infty,0]$, decreasing on $[0,+\infty)$ and $\max_{x\in\RR}
m_\xi(x)=m_\xi(0)=\frac12E|\xi_1|$.
 We will need the following assumption.

\begin{itemize}
\item[\bf (A)]
{\it $E(\xi_1^2)=\sigma^2<\infty$ and the measure
$m_\xi(z)\,dz$ is absolutely continuous with respect to $dF_\xi(z)$
with a bounded Radon-Nikodym derivative, i.e., there exists a
function $g_\xi :\RR\to\RR_+$ such that
$\|g_\xi\|_\infty\triangleq\sup_{x\in\RR} g_\xi(x)<\infty$ and
$$
\int_a^{a'} m_\xi(z)\,dz=\int_a^{a'} g_\xi(z)\,dF_\xi(z),\qquad \forall a,a'\in\RR.
$$
}
\end{itemize}
Clearly, Assumption (A) is a restriction on the probability
distribution of the errors $\xi_i$. Some examples where Assumption
(A) is fulfilled are:
\begin{itemize}
\item[(i)] If $\xi_1\sim \mathcal N(0,\sigma^2)$, then $g_\xi(x)\equiv \sigma^2$.
\item[(ii)] If $\xi_1$ is uniformly distributed in the interval $[-b,b]$, then $m_\xi(x)=(b^2-x^2)_+/(4b)$ and $g_\xi(x)=(b^2-x^2)_+/2$.
\item[(iii)] If $\xi_1$ has a density function $f_\xi$ with compact
support $[-b,b]$ and such that $f_\xi(x)\ge f_{\text{min}}>0$ for
every $x\in[-b,b]$, then assumption (A) is satisfied with
$g_\xi(x)=m_\xi(x)/f_\xi(x)\le E|\xi_1|/(2f_{\text{min}})$.
\end{itemize}
We now give some examples where (A) is not fulfilled:
\begin{itemize}
\item[(iv)] If $\xi_1$ has a double exponential distribution with
zero mean and variance $\sigma^2$, then $g_\xi(x)= (\sigma^2+\sqrt{2\sigma^2}|x|)/2$.
\item[(v)] If $\xi_1$ is a Rademacher random variable, then
$m_\xi(x)=\1(|x|\le 1)/2$, and the measure $m_\xi(x)dx$ is not
absolutely continuous with respect to the distribution of $\xi_1$.
\end{itemize}
The following lemma can be viewed as an extension of Stein's
identity (cf.\ \cite{lemcas}).
\begin{lemma}\label{lem1} Let $T_n(x,\bY)$ be an estimator of $f(x)$
such that the mapping $\bY\mapsto T_n(\bY)\triangleq
(T_n(x_1,\bY),\ldots,T_n(x_n,\bY))^\T$ is continuously
differentiable and let us denote by $\partial_j T_n(x_i,\bY)$ the
partial derivative of the function $\bY\mapsto T_n(\bY)$ with
respect to the $j$th coordinate of $\bY$. If Assumption (A) and the
following condition
\begin{equation}\label{cond1}
\begin{matrix}
\displaystyle\int_\RR|y|\int_0^y\,|\partial_i
T_n(x_i,f+\bz)|\,dz_i\,dF_\xi(y)<\infty,  \quad
i=1,\dots,n,\\
\text{or}\\
\partial_i T_n(x_i,\bY)\ge 0,\qquad \forall\bY\in\RR^n,\quad \,i=1,\ldots,n,
\phantom{\text{$\displaystyle\int$}}
\end{matrix}
\end{equation}
are satisfied where $\bz=(z_1,\dots,z_n)^\T$,
$f=(f(x_1),\dots,f(x_n))^\T$ then
$$
E[\hat r_n(\bY)]=E\Big(\|T_n(\bY)-f\|_n^2\Big),
$$
where
$$
\hat r_n(\bY)=\|T_n(\bY)-\bY\|_n^2+\frac{2}n\sum_{i=1}^n
\partial_i T_n(x_i,\bY)\,g_\xi(\xi_i)-\sigma^2.
$$
\end{lemma}

\begin{proof} We have
\begin{eqnarray}
E\Big(\|T_n(\bY)-f\|_n^2\Big)&=&
E\Big[\|T_n(\bY)-\bY\|_n^2+\frac{2}n\sum_{i=1}^n \xi_i
(T_n(x_i,\bY)-f(x_i))\Big]-\sigma^2\nonumber\\
&=& E\Big[\|T_n(\bY)-\bY\|_n^2+\frac{2}n\sum_{i=1}^n \xi_i
T_n(x_i,\bY)\Big]-\sigma^2.\label{l11}
\end{eqnarray}
For $\bz=(z_1,\ldots,z_n)^\T\in\RR^n$ write
$F_{\xi,i}(\bz)=\prod_{j\not=i}F_\xi(z_j)$. Since $E(\xi_i)=0$ we
have
\begin{align}
E[\xi_i T_n(x_i,\bY)]&=E\Big[\xi_i\int_0^{\xi_i}
\partial_i T_n(x_i,Y_1,\ldots,Y_{i-1},f(x_i)+z,Y_{i+1},\ldots,Y_n)\,dz\Big]\nonumber\\
&=\int_{\RR^{n-1}}\Big[\int_\RR y\int_0^y\,\partial_i T_n(x_i,f+\bz)\,dz_i\,dF_\xi(y)\Big]
\,dF_{\xi,i}(\bz).\label{l12}
\end{align}
Condition (\ref{cond1}) allows us to apply the Fubini theorem to the
expression in squared brackets on the right hand side of the last
display. Thus, using the definition of $m_\xi$ and Assumption (A) we
find
\begin{align*}
\int_{\RR_+} y\int_0^y\,\partial_i T_n(x_i,f+\bz)\,dz_i\,dF_\xi(y)&=\int_{\RR_+}\int_{z_i}^\infty y\,dF_\xi(y)\,\partial_i T_n(x_i,f+\bz)\,dz_i\\
&=\int_{\RR_+} m_\xi(z_i)\,\partial_i T_n(x_i,f+\bz)\,dz_i\\
&=\int_{\RR_+} g_\xi(z_i)\,\partial_i T_n(x_i,f+\bz)\,dF_\xi(z_i).
\end{align*}
Similar equality holds for the integral over $\RR_-$. Thus, in view
of (\ref{l12}), we obtain
\begin{eqnarray*}
E[\xi_i T_n(x_i,\bY)]=E[\partial_i T_n(x_i,\bY)\,g_\xi(\xi_i)].
\end{eqnarray*}
Combining the last display with (\ref{l11}) we get the lemma.
\end{proof}

Based on Lemma \ref{lem1}
we obtain the following bound on the risk of the
exponentially weighted aggregate $\hat f_n$.

\begin{theorem}\label{thm6}
Let $\pi$ be an element of $\scrP_\Lambda$ such that, for all
$\bY'\in \RR^n$ and $\beta>0$, the mappings
$\lambda\mapsto\theta_\lambda(\bY') f_\lambda^2(x_i)$,
$i=1,\dots,n$, are $\pi$-integrable. If Assumption (A) is fulfilled
then the aggregate $\hat f_n$ defined by (\ref{aggr}) with
$\beta\geq 4\|g_\xi\|_\infty$ satisfies the inequality
\begin{equation}\label{oracle6}
E\Big(\|\hat f_n-f\|_n^2\Big)\le \int
\|f_\lambda-f\|_n^2\,p(d\lambda)+\frac{\beta\, \mathcal
K(p,\pi)}{n}, \quad \forall \ p\in \scrP_\Lambda,
\end{equation}
where $\KL(p,\pi)$ stands for the Kullback-Leibler divergence
between $p$ and $\pi$.
\end{theorem}

\begin{proof}  We will now use Lemma 1 with $T_n=\hat f_n$. Accordingly,
we write here $\hat f_n(x_i,\bY)$ instead of $\hat f_n(x_i)$.
Applying the dominated convergence theorem and taking into
account the definition of $\theta_\lambda(\bY)$ we easily find
 that the
$\pi$-integrability of
$\lambda\mapsto\theta_\lambda(\bY')
f_\lambda^2(x_i)$ for all $i$, $\bY'$ implies that the mapping $\bY\mapsto
\hat f_n(\bY)\triangleq (\hat f_n(x_1,\bY),\ldots,\hat
f_n(x_n,\bY))^\T$ is continuously differentiable. Simple algebra yields
\begin{align*}
\partial_i \hat f_n(x_i,\bY)&=\frac{2}\beta
\bigg\{\int_\Lambda  f_\lambda^2(x_i)\theta_\lambda(\bY)
\pi(d\lambda)-\hat f_n^2(x_i,\bY)\bigg\}\\
&=\frac{2}{\beta}\int_\Lambda (f_\lambda(x_i)-\hat
f_n(x_i,\bY))^2\,\theta_\lambda(\bY)\,\pi(d\lambda) \ge 0.
\end{align*}
Therefore, (\ref{cond1}) is fulfilled for $T_n=\hat f_n$
and we can apply Lemma \ref{lem1} which yields
$$E[\hat
r_n(\bY)]=E\Big(\|\hat f_n(\bY)-f\|_n^2\Big)$$ with
$$
\hat r_n(\bY)=\|\hat f_n(\bY)-\bY\|_n^2+\frac{2}n\sum_{i=1}^n \partial_i \hat f_n(x_i,\bY)\,g_\xi(\xi_i)-\sigma^2.
$$
Since $\hat f_n(\bY)$ is the expectation of $f_\lambda$ w.r.t.
the probability measure $\theta\cdot\pi$,
\begin{align*}
\|\hat f_n(\bY)-\bY\|_n^2&=\int_\Lambda \{ \|f_\lambda-\bY\|_n^2-\|f_\lambda-\hat f_n(\bY)\|_n^2\}\,\theta_\lambda(\bY)\,\pi(d\lambda).
\end{align*}
Combining these results we get
\begin{align*}
\hat r_n(\bY)&=\int_\Lambda \big\{ \|f_\lambda-\bY\|_n^2- \hbox{$
\sum_{i=1}^n\frac{(\beta-4g_\xi(\xi_i))(f_\lambda(x_i)-\hat
f_n(x_i,\bY))^2}{n\beta} $}
\big\}\,\theta_\lambda(\bY)\,\pi(d\lambda)-\sigma^2\\
&\le \int_\Lambda \|f_\lambda-\bY\|_n^2\,\theta_\lambda(\bY)\,\pi(d\lambda)-\sigma^2,
\end{align*}
where we used that $\beta\ge 4 \|g_\xi\|_\infty$. By definition of
$\theta_\lambda$,
$$-n\|f_\lambda-\bY\|_n^2\,
=\beta\log\theta_\lambda(\bY)+\beta \log\big[\int_\Lambda
e^{-n\|\bY-f_w\|_n^2/\beta}\,\pi(dw)\big].
$$
Integrating this equation over $\theta\cdot \pi$, using the fact
that $\int_\Lambda
\theta_\lambda(\bY)\,\log\theta_\lambda(\bY)\,\pi(d\lambda)
=\KL(\theta\cdot\pi,\pi)\ge~0$ and convex duality argument (cf.,
e.g., \cite{dz}, p.264, or \cite{catbook:01}, p.160) we get
\begin{align*}
\hat r_n(\bY)&\le
-\frac{\beta}n\log\bigg[\int_\Lambda e^{-n\|\bY-f_w\|_n^2/\beta}\,\pi(dw)\bigg]-\sigma^2\\
&\le \int_{\Lambda}\|\bY-f_w\|_n^2\, p(dw)+\frac{\beta\KL(p,\pi)}n-
\sigma^2
\end{align*}
for all $p\in \scrP_\Lambda$. Taking expectations in the last
inequality we obtain (\ref{oracle6}).
\end{proof}

\section{Risk bounds for $n$-divisible distributions of errors}
\label{sec3}

In this section we present a second approach to prove sharp risk
bounds of the form (\ref{oracle6}). The main idea of the
proof consists in an artificial introduction of a ``dummy"
random vector $\zzeta\in \RR^n$ independent of
$\xxi=(\xi_1,\ldots,\xi_n)$ and having the same type of distribution as
$\xxi$. This approach will allow us to cover the class of
distributions of $\xi_i$ satisfying the following assumption.
\begin{itemize}
\item[\bf (B)] {\it There exist i.i.d.\ random variables $\zeta_1,\ldots,\zeta_n$
defined on an enlargement of the probability space
$(\Omega,\calF,P)$ such that:\vspace{1mm}
\begin{itemize}
\item[(B1)] the random variable $\xi_1+\zeta_1$ has the same
distribution as $(1+1/n)\xi_1$,\vspace{1mm}
\item[(B2)] the vectors $\zzeta=(\zeta_1,\ldots,\zeta_n)$ and
$\xxi=(\xi_1,\ldots,\xi_n)$ are independent.
\end{itemize}
}
\end{itemize}
If $\xi_1$ satisfies (B1), then we will say that its distribution is
\textit{$n$-divisible}.

We will need one more assumption. Let
$L_\zeta:\RR\to\RR\cup\{\infty\}$ be the moment generating function
of the random variable $\zeta_1$, i.e.,
$L_\zeta(t)=E(e^{t\zeta_1})$, $t\in\RR$.
\begin{itemize}
\item[\bf (C)]{\it
There exist a functional $\Psi_\beta:\pp\times\pp\to\RR$ and a real
number $\beta_0>0$ such that
\begin{align}\label{Psi}
\ \qquad\begin{cases}
e^{(\|f-\f_{\mu'}\|_n^2-\|f-\f_{\mu}\|_n^2)/\beta}
\prod_{i=1}^{n}L_{\zeta}\Big(\frac{2(\f_{\mu}(x_i)-\f_{\mu'}(x_i))}{\beta}\Big)
\leq \Psi_\beta(\mu,\mu'),\\
\mu\mapsto \Psi_\beta(\mu,\mu')\ \text{is concave and continuous in
the total }\\\text{variation norm for any }
\mu'\in\pp,\\
\Psi_\beta(\mu,\mu)=1,
\end{cases}
\end{align}
for any $\beta\ge\beta_0$.
}
\end{itemize}
We now discuss some sufficient conditions for assumptions (B) and
(C). Denote by $\mathcal D_n$ the set of all probability
distributions of $\xi_1$ satisfying assumption (B1). First, it is
easy to see that all the zero-mean Gaussian or double exponential
distributions belong to $\mathcal D_n$. Furthermore, $\mathcal D_n$
contains all the stable distributions. However, since the non-Gaussian
stable distributions do not have second order moments, they do not
satisfy (\ref{Psi}). One can also check that the convolution of two
distributions from $\mathcal D_n$ belongs to $\mathcal D_n$.
Finally, note that the intersection $\mathcal D=\cap_{n\geq
1}\mathcal D_n$ is included in the set of all infinitely divisible
distributions and is called the L-class (see \cite{p95},
Theorem~3.6, p.~102).

However, some basic distributions such as the uniform or the
Bernoulli distribution do not belong to $\mathcal D_n$. To show
this, let us recall that the characteristic function of the uniform
on $[-a,a]$ distribution is given by $\varphi(t)=\sin(at)/(\pi at)$.
For this function, $\varphi((n+1)t)/\varphi(nt)$ is equal to
infinity at the points where $\sin(nat)$ vanishes (unless $n=1$).
Therefore, it cannot be a characteristic function. Similar argument
shows that the centered Bernoulli and centered binomial
distributions do not belong to $\mathcal D_n$.

Assumption (C) can be readily checked when the
moment generating function $L_\zeta(t)$ is locally sub-Gaussian,
i.e., there exists a constant $c>0$ such that the inequality
$L_\zeta(t)\le e^{ct^2}$ holds for sufficiently small values of $t$.
Examples include all the zero-mean distributions with bounded
support, the Gaussian and double-exponential distributions, etc. The
validity of Assumption (C) for such distributions follows from
Lemma~\ref{lem3} in the Appendix.

\begin{theorem}\label{thm1}
Let $\pi$ be an element of $\scrP_\Lambda$ such that, for all
$\bY'\in \RR^n$ and $\beta>0$, the mappings
$\lambda\mapsto\theta_\lambda(\bY') f_\lambda(x_i)$, $i=1,\dots,n$,
are $\pi$-integrable. If assumptions (B) and (C) are fulfilled, then
the aggregate $\hat f_n$ defined by (\ref{aggr}) with
$\beta\geq\beta_0$ satisfies the inequality
\begin{equation}\label{oracle}
E\Big(\|\hat f_n-f\|_n^2\Big)\le \int
\|f_\lambda-f\|_n^2\,p(d\lambda)+\frac{\beta\, \mathcal
K(p,\pi)}{n+1}, \quad \forall \ p\in \scrP_\Lambda.
\end{equation}
\end{theorem}

\begin{proof}
Define the mapping $\bH:\pp\to\RR^n$ by
$$\bH_\mu=(\f_\mu(x_1)-f(x_1),\ldots,\f_\mu(x_n)-f(x_n))^\T,\quad
\mu\in\pp.$$ For brevity, we will write
$$\bh_\lambda=\bH_{\delta_\lambda}=(f_\lambda(x_1)-f(x_1),\ldots,
f_\lambda(x_n)-f(x_n))^\T, \quad \lambda\in\Lambda,$$ where
$\delta_\lambda$ is the Dirac measure at $\lambda$ (that is
$\delta_\lambda(A)=\1(\lambda\in A)$ for any $A\in\calA$).

Since $E(\xi_i)=0$, assumption (B1) implies that $E(\zeta_i)=0$ for
$i=1,\ldots,n$. On the other hand, (B2) implies that $\zzeta$ is
independent of $\theta_\lambda$. Therefore, we have
\begin{align}\label{1new1}
E\Big(\|\f_{\theta\cdot\pi}-f\|_n^2\Big) &= \beta
E\log\exp\Big\{
\frac{\|\f_{\theta\cdot\pi}-f\|_n^2-2\zzeta^\T\bH_{\theta\cdot\pi}}
{\beta}\Big\}
=S+S_1
\end{align}
where
\begin{align*}
&S=-\beta E\log\int_\Lambda \theta_\lambda\exp\Big\{-
\frac{\|f_\lambda-f\|_n^2-2\zzeta^\T \bh_{\lambda}}{\beta}\Big\}
\pi(d\lambda),\\
&S_1=\beta E\log\int_\Lambda  \theta_\lambda \exp\Big\{
\frac{\|\f_{\theta\cdot\pi}-f\|_n^2-
\|f_\lambda-f\|_n^2
+2\zzeta^\T(\bh_{\lambda}-\bH_{\theta\cdot\pi})}{\beta}\Big\}\pi(d\lambda).
\end{align*}
The definition of $\theta_\lambda$ yields
\begin{align}\label{eq1S}
S&=-\beta E\log\int_\Lambda\exp\Big\{-
\frac{n\|\bY-f_\lambda \|_n^2+\|f_\lambda -f\|_n^2-2\zzeta^\T\bh_{\lambda}}{\beta}\Big\}\pi(d\lambda)\nonumber\\
&\quad+\beta E\log\int_\Lambda \exp\Big\{- \frac{n\|\bY-f_\lambda
\|_n^2}{\beta}\Big\}\pi(d\lambda).
\end{align}
Since $\|\bY-f_\lambda
\|_n^2=\|\xxi\|_n^2-2n^{-1}\xxi^\T\bh_{\lambda}+\|f_\lambda
-f\|_n^2$, we get
\begin{align}\label{eq2S}
S&=-\beta E\log\int_\Lambda \exp\Big\{- \frac{(n+1)\|f_\lambda
-f\|_n^2 -2(\xxi+\zzeta)^\T
\bh_{\lambda}}{\beta}\Big\}\pi(d\lambda)\nonumber\\
&\quad+\beta E\log\int_\Lambda \exp\Big\{- \frac{n\|f-f_\lambda
\|_n^2-2\xxi^\T \bh_{\lambda}}{\beta}\Big\}
\pi(d\lambda)\nonumber\\
&=\beta E\log \int_\Lambda  e^{-n\rho(\lambda)}\pi(d\lambda)- \beta
E\log \int_\Lambda  e^{-(n+1)\rho(\lambda)}\pi(d\lambda),
\end{align}
where we used the notation $\rho(\lambda)=(\|f-f_\lambda
\|_n^2-2n^{-1}\xxi^\T \bh_{\lambda})/\beta$ and the fact that
$\xxi+\zzeta$ can be replaced by $(1+1/n)\xxi$ inside the
expectation. The H\"older inequality implies that $\int_\Lambda
e^{-n\rho(\lambda)}\pi(d\lambda)\leq (\int_\Lambda
e^{-(n+1)\rho(\lambda)}\pi(d\lambda))^ \frac{n}{n+1}$. Therefore,
\begin{equation}\label{S2}
S\leq -\frac\beta{n+1} E\log \int_\Lambda
e^{-(n+1)\rho(\lambda)}\,\pi(d\lambda).
\end{equation}

Assume now that $p\in\scrP_\Lambda$ is absolutely continuous with
respect to $\pi$. Denote by $\phi$ the corresponding Radon-Nikodym
derivative and by $\Lambda_+$ the support of $p$. Using the
concavity of the logarithm and Jensen's inequality we get
\begin{align*}
-E\log\int_\Lambda e^{-(n+1)\rho(\lambda)}\pi(d\lambda)&\le -E\log\int_{\Lambda_+} e^{-(n+1)\rho(\lambda)}\pi(d\lambda)\\
&=-E\log\int_{\Lambda_+} e^{-(n+1)\rho(\lambda)}\phi^{-1}(\lambda)\,p(d\lambda)\\
&\le
(n+1)E\int_{\Lambda_+}\rho(\lambda)\,p(d\lambda)+\int_{\Lambda_+}
\log \phi(\lambda)\,p(d\lambda).
\end{align*}
Noticing that the last integral here equals to $\mathcal K(p,\pi)$
and combining the resulting inequality with (\ref{S2}) we obtain
$$
S\leq \beta E\int_\Lambda
\rho(\lambda)\,p(d\lambda)+\frac{\beta\,\mathcal K(p,\pi)}{n+1}.
$$
Since $E(\xi_i)=0$ for every $i=1,\ldots,n$, we have $\beta
E(\rho(\lambda))=\|f_\lambda-f\|_n^2$, and using the Fubini theorem
we find
\begin{align}\label{Seq}
S\leq \int_\Lambda
\|f_\lambda-f\|_n^2\,p(d\lambda)+\frac{\beta\,\mathcal
K(p,\pi)}{n+1}\,.
\end{align}
Note that this inequality also holds  in the case where $p$ is not
absolutely continuous with respect to $\pi$, since in this case
$\mathcal K(p,\pi)=\infty$.

To complete the proof, it remains to show that $S_1\le 0$. Let
$E_{\xxi}(\cdot)$ denote the conditional expectation
$E(\cdot|\xxi)$. By the concavity of the logarithm,
$$
S_1\le\beta E\log\int_\Lambda \theta_\lambda E_{\xxi}\exp\Big\{
\frac{\|\f_{\theta\cdot\pi}-f\|_n^2-\|f_{\lambda}-f\|_n^2+
2\zzeta^\T(\bh_{\lambda}-\bH_{\theta\cdot\pi})}{\beta}\Big\}\pi(d\lambda).
$$
Since $f_\lambda=\f_{\delta_\lambda}$ and $\zzeta$ is independent of
$\theta_\lambda$, the last expectation on the right hand side of
this inequality is bounded from above by
$\Psi_\beta(\delta_\lambda,\theta\cdot\pi)$. Now, the  fact that
$S_1\le 0$ follows from the concavity and continuity of the
functional $\Psi_\beta(\cdot,\theta\cdot\pi)$, Jensen's
inequality and the equality
$\Psi_\beta(\theta\cdot\pi,\theta\cdot\pi)=1$.
\end{proof}

Another way to read the results of Theorems~\ref{thm6} and
\ref{thm1} is that, if the ``phantom'' Gaussian error model
(\ref{phm}) with variance taken larger than a certain threshold
value is used to construct the Bayesian posterior mean $\hat f_n$,
then $\hat f_n$ is close on the average to the best prediction under
the true model, even when the true data generating distribution is
non-Gaussian.
%%%

We now illustrate application of Theorem \ref{thm1} by an example.
Assume that the errors $\xi_i$ are double exponential, that is the
distribution of $\xi_i$ admits a density with respect to the
Lebesgue measure given by
$$
f_\xi(x)=\frac1{\sqrt{2\sigma^2}}\;e^{-\sqrt{2}|x|/\sigma},\quad
x\in\RR.
$$
Aggregation under this assumption is discussed in \cite{y03} where
it is recommended to modify the weights (\ref{ag1}) matching them to
the shape of $f_\xi$. For such a procedure \cite{y03} proves an
oracle inequality with leading constant which is greater than 1. The
next proposition shows that sharp risk bounds (i.e., with leading
constant 1) can be obtained without modifying the weights
(\ref{ag1}).

\begin{proposition}\label{prop2}
Assume that $\sup_{\lambda\in\Lambda}\|f-f_\lambda \|_n\le L<\infty$
and $\sup_{i,\lambda}|f_\lambda(x_i)|\le \bar L<\infty$. Let the
random variables $\xi_i$ be i.i.d.\ double exponential with variance
$\sigma^2>0$. Then for any $\beta$ larger than
$$
\max\bigg( \bigg(8+\frac4n\bigg)\sigma^2+2L^2 ,\
4\sigma\bigg(1+\frac1n\bigg){\bar L} \bigg)
$$ the aggregate $\hat f_n$ satisfies inequality (\ref{oracle}).
\end{proposition}

\begin{proof}
We apply Theorem~\ref{thm1}. The characteristic function of the
double exponential density is $\varphi(t)=2/ (2+\sigma^2t^2)$.
Solving $\varphi(t)\varphi_\zeta(t)=\varphi((n+1)t/n)$ we get the
characteristic function $\varphi_\zeta$ of $\zeta_1$. The
corresponding Laplace transform $L_\zeta$ in this case is
$L_\zeta(t)=\varphi_\zeta(-{\rm i}t)$, which yields
$$
L_\zeta(t)=1+\frac{(2n+1)\sigma^2 t^2}{2n^2-(n+1)^2\sigma^2 t^2}.
$$
Therefore
$$
\log L_\zeta(t)\leq (2n+1)(\sigma t/n)^2, \quad |t|\leq
\frac{n}{(n+1)\sigma}\,.
$$
We now use this inequality to check assumption (C). Let $\beta$ be
larger than $4\sigma\big(1+1/n\big){\bar L}$. Then for all
$\mu,\mu'\in\scrP_\Lambda$ we have
$$
\frac{2\big|\f_{\mu}(x_i)-\f_{\mu'}(x_i)\big|}{\beta}\le
\frac{4{\bar L}}\beta\le \frac{n}{(n+1)\sigma}, \quad i=1,\dots,n,
$$
and consequently
$$
\log
L_\zeta\bigg(2\big|\f_{\mu}(x_i)-\f_{\mu'}(x_i)\big|/\beta\bigg)\leq
\frac{4\sigma^2(2n+1)(\f_{\mu}(x_i) -\f_{\mu'}(x_i))^2}{n^2\beta^2}\
.
$$
This implies that
$$
\exp{\Big(\frac{\|f-\f_{\mu'}\|_n^2-\|f-\f_{\mu}\|_n^2}{\beta}\Big)}
\prod_{i=1}^{n}L_{\zeta}\Big(\frac{2(\f_{\mu}(x_i)-\f_{\mu'}(x_i))}{\beta}\Big)
\leq \Psi_\beta(\mu,\mu'),
$$
where
$$
\Psi_\beta(\mu,\mu')=\exp\bigg(\frac{\|f-\f_{\mu'}\|_n^2-\|f-\f_{\mu}
\|_n^2}{\beta}+\frac{4\sigma^2(2n+1)\|\f_{\mu}
-\f_{\mu'}\|_n^2}{n\beta^2} \bigg).
$$
This functional satisfies $\Psi_\beta(\mu,\mu)=1$, and it is not
hard to see that the mapping $\mu\mapsto\Psi_\beta(\mu,\mu')$ is
continuous in the total variation norm. Finally, this mapping is
concave for every $\beta\geq
(8+4/n)\sigma^2+2\sup_\lambda\|f-f_\lambda \|^2_n$ by virtue of
Lemma~\ref{lem3} in the Appendix. Therefore, assumption (C) is
fulfilled and the desired result follows from Theorem~\ref{thm1}.

\end{proof}

An argument similar to that of Proposition \ref{prop2} can be used
to deduce from Theorem \ref{thm1} that if the random variables
$\xi_i$ are i.i.d.\ Gaussian $\mathcal N(0,\sigma^2)$, then
inequality (\ref{oracle}) holds for every $\beta\geq
(4+2/n)\sigma^2+2L^2$ (cf.\ \cite{dt07}). However, in this
Gaussian framework we can also apply Theorem \ref{thm6} that
gives better result: essentially the same inequality (the only
difference is that the Kullback divergence is divided by $n$ and
not by $n+1$) holds for $\beta\geq 4\sigma^2$, with no assumption
on the function $f$.

\section{Model selection with finite or countable $\Lambda$}
\label{sec4}

Consider now the particular case where $\Lambda$ is countable.
W.l.o.g. we suppose that $\Lambda=\{1,2,\dots\}$, $\{f_\lambda,
\lambda\in\Lambda\}=\{f_j\}_{j=1}^\infty$ and we set
$\pi_j\triangleq\pi(\lambda=j)$. As a corollary of
Theorem~\ref{thm1} we get the following sharp oracle inequalities
for model selection type aggregation.
\begin{theorem}\label{modsel}
Let either assumptions of Theorem \ref{thm6} or those of Theorem
\ref{thm1} be satisfied and let $\Lambda$ be countable. Then for any
$\beta\ge \beta_0$ the aggregate $\hat f_n$ satisfies the inequality
$$
E\Big(\|\hat f_n-f\|_n^2\Big)\le \inf_{j\ge1}
\bigg(\|f_j-f\|_n^2+\frac{\beta\, \log\pi_j^{-1}}{n}\bigg)
$$
where $\beta_0=4\|g_\xi\|_\infty$ when Theorem \ref{thm6} is
applied. In particular, if $\pi_j=1/M$, $j=1,\dots,M$, we have, for
any $\beta\ge \beta_0$,
\begin{equation}\label{ms}
E\Big(\|\hat f_n-f\|_n^2\Big)\le \min_{j=1,\ldots,M}
\|f_j-f\|_n^2+\frac{\beta\, \log M}{n}.
\end{equation}
\end{theorem}
\begin{proof}
For a fixed integer $j_0\ge1$ we apply Theorems~\ref{thm6} or
\ref{thm1} with $p$ being the Dirac measure: $p(\lambda=j)=\1
(j=j_0),\, j\ge1$. This gives
$$
E\Big(\|\hat f_n-f\|_n^2\Big)\le  \|f_{j_0}-f\|_n^2+\frac{\beta\,
\log\pi_{j_0}^{-1}}{n}.
$$
Since this inequality holds for every $j_0$, we obtain the first
inequality of the proposition. The second inequality is an obvious
consequence of the first one.
\end{proof}

Theorem \ref{modsel} generalizes the result of \cite{lb06} where
the case of finite $\Lambda$ and Gaussian errors $\xi_i$ is
treated. For this case it is known that the rate of convergence
$(\log M)/n$ in (\ref{ms}) cannot be improved
\cite{tsy:03,btw07a}. Furthermore, for the examples (i) -- (iii)
of Section \ref{sec2} (Gaussian or bounded errors) and finite
$\Lambda$, inequality (\ref{ms}) is valid with no assumption on
$f$ and $f_\lambda$. Indeed, when $\Lambda$ is finite the
integrability conditions are automatically satisfied. Note that,
for bounded errors $\xi_i$, oracle inequalities of the form
(\ref{ms}) are also established in the theory of prediction of
deterministic sequences \cite{v:90,lw94,cetal, kw99, lcb:06}.
However, those results require uniform boundedness not only of the
errors $\xi_i$ but also of the functions $f$ and $f_\lambda$.
What is more, the minimal allowed values of $\beta$ in those
works depend on an upper bound on $f$ and $f_\lambda$ which is not
always available. The version of (\ref{ms}) based on Theorem
\ref{thm6} is free of such a dependence.

%%%%%%%%%%%%%%%%%%%%

\section{Risk bounds for general distributions of errors}
\label{sec5}

As discussed above, assumption (B) restricts the application of
Theorem \ref{thm1} to models with $n$-divisible errors. We now
show that this limitation can be dropped. The main idea of the
proof of Theorem~\ref{thm1} was to introduce a dummy random vector
$\zzeta$ independent of $\xxi$. However, the independence
property is stronger than what we really need in the proof of
Theorem~\ref{thm1}. Below we come to a weaker condition invoking
a version of the Skorokhod embedding (a detailed survey on this
subject can be found in \cite{obloj}).

For simplicity we assume that the errors $\xi_i$ are symmetric,
i.e., $P(\xi_i>a)=P(\xi_i<-a)$ for all $a\in\RR$. The argument can
be adapted to the asymmetric case as well, but we do not discuss it
here.

First, we describe a version of Skorokhod's construction that will
be used below, cf.\ \cite[Proposition II.3.8]{RevYor}.

\begin{lemma}\label{l1} Let $\xi_1,\ldots,\xi_n$ be i.i.d. symmetric random variables
on $(\Omega,\calF,P)$. Then there exist i.i.d.\ random variables
$\zeta_1,\ldots,\zeta_n$ defined on an enlargement of the
probability space $(\Omega,\calF,P)$ such that
\begin{itemize}
\item[(a)] $\xxi+\zzeta$ has the same distribution as $(1+1/n)\xxi$.
\item[(b)] $E(\zeta_i|\xxi)=0$, $i=1,\dots,n$,
\item[(c)]
for any $\lambda>0$ and for any $i=1,\ldots,n$, we have
$$
E(e^{\lambda\zeta_i}|\xxi)\le e^{(\lambda\xi_i)^2(n+1)/n^2}.
$$
\end{itemize}
\end{lemma}

\begin{proof} Define $\zeta_i$ as a random variable such that, given
$\xi_i$, it takes values $\xi_i/n$ or $-2\xi_i-\xi_i/n$ with
conditional probabilities $P(\zeta_i=\xi_i/n|\xi_i)=(2n+1)/(2n+2)$
and $P(\zeta_i=-2\xi_i-\xi_i/n|\xi_i)=1/(2n+2)$. Then properties (a)
and (b) are~straightforward. Property (c) follows from the relation
$$
E(e^{\lambda\zeta_i}|\xi_i)=e^{\frac{\lambda \xi_i}n}\bigg(1+
\frac{1}{2n+2}\big(e^{-2\lambda \xi_i(1+1/n)}-1\big)\bigg)
$$
and Lemma~\ref{lem2} in the Appendix with $x=\lambda\xi_i/n$ and
$\alpha_0=2n+2$.
\end{proof}

We now state the main result of this section.

\begin{theorem}\label{thm3} Fix some $\alpha>0$ and assume that
$\sup_{\lambda\in\Lambda}\|f-f_\lambda\|_n\le L$ for a finite
constant $L$. If the errors $\xi_i$ are symmetric and have a finite
second moment $E(\xi_i^2)$, then for any $\beta\geq
4(1+1/n)\alpha+2L^2$ we have
\begin{equation}\label{oracle1}
E\Big(\|\hat f_n-f\|_n^2\Big)\le \int_\Lambda \|f_\lambda-f\|_n^2\,
p(d\lambda)+\frac{\beta\, \mathcal K(p,\pi)}{n+1}+R_n, \quad \forall
\ p\in \scrP_\Lambda,
\end{equation}
where the residual term $R_n$ is given by
$$
R_n=E^*\bigg(\sup_{\lambda\in\Lambda} \sum_{i=1}^n
\frac{4(n+1)(\xi^2_i-\alpha)(f_\lambda(x_i)-
\f_{\theta\cdot\pi}(x_i))^2}{n^2\beta} \bigg)
$$
and $E^*$ denotes the expectation with respect to the outer
probability.
\end{theorem}
\begin{proof}
We slightly modify the proof of Theorem~\ref{thm1}. We now consider
a dummy random vector $\zzeta=(\zeta_1,\ldots,\zeta_n)$ as in
Lemma~\ref{l1}. Note that for this $\zzeta$ relation (\ref{1new1})
remains valid: in fact, it suffices to condition on $\xxi$, to use
Lemma~\ref{l1}(b) and the fact that $\theta_\lambda$ is measurable
with respect to $\xxi$. Therefore, with the notation of the proof
of Theorem~\ref{thm1}, we have $E(\|\hat f_n-f\|^2_n)=S+S_1$. Using
Lemma~\ref{l1}(a) and acting exactly as in the proof of
Theorem~\ref{thm1} we get that $S$ is bounded as in (\ref{Seq}).
Finally, as shown in the proof of Theorem~\ref{thm1} the term $S_1$
satisfies
$$
S_1\le\beta E\log\int_\Lambda \theta_\lambda E_{\xxi}\exp\Big\{
\frac{\|\f_{\theta\cdot\pi}-f\|_n^2-\|f_{\lambda}-f\|_n^2+
2\zzeta^\T(\bh_{\lambda}-\bH_{\theta\cdot\pi})}{\beta}\Big\}\pi(d\lambda).
$$
According to Lemma~\ref{l1}(c),
$$
E_{\xxi} \Big(e^{2\zzeta^T
(\bh_{\lambda}-\bH_{\theta\cdot\pi})/\beta}\Big)\le
\exp\bigg\{\sum_{i=1}^n
\frac{4(n+1)(f_\lambda(x_i)-\f_{\theta\cdot\pi}(x_i))^2\xi_i^2}{n^2\beta^2}\bigg\}.
$$
Therefore, $S_1 \le S_2 + R_n$, where
$$
S_2=\beta
E\log\!\int_\Lambda\!\theta_\lambda\!\exp\Big(\frac{4\alpha(n+1)\|f_{\lambda}
-\f_{\theta\cdot\pi}\|_n^2}{n\beta^2}-
\frac{\|f-f_\lambda\|_n^2-\|f-\f_{\theta\cdot\pi}\|_n^2}{\beta}
\Big)\pi(d\lambda).
$$
Finally, we apply Lemma~\ref{lem3} (cf.\ Appendix) with $s^2=4\alpha(n+1)$ and
Jensen's inequality to get that $S_2\le0$.
\end{proof}

In view of Theorem \ref{thm3}, to get the bound (\ref{oracle}) it
suffices to show that the remainder term $R_n$ is non-positive under
some assumptions on the errors $\xi_i$. More generally, we may
derive somewhat less accurate inequalities than (\ref{oracle}) by
proving that $R_n$ is small enough. This is illustrated by the
following corollaries.

\begin{corollary}\label{cor0}
Let the assumptions of Theorem \ref{thm3} be satisfied and let
$|\xi_i|\le B$ almost surely where $B$ is a finite constant. Then
the aggregate $\hat f_n$ satisfies inequality (\ref{oracle}) for any
$\beta\geq 4B^2(1+1/n)+2L^2$.
\end{corollary}

\begin{proof}
It suffices to note that for $\alpha=B^2$ we get $R_n\le 0$.
\end{proof}

\begin{corollary}\label{cor1} Let the assumptions of
Theorem \ref{thm3} be satisfied and suppose that
$E(e^{t|\xi_i|^\kappa})\le B$ for some constants $t>0$, $\kappa>0$,
$B>0$. Then for any $n\geq e^{1/\kappa}$ and any $\beta\geq
4(1+1/n)(2(\log n)/t)^{2/\kappa}+2L^2$ we have
\begin{eqnarray}\label{oracle3}
E\Big(\|\hat f_n-f\|_n^2\Big)&\le& \int_\Lambda
\|f_\lambda-f\|_n^2\,
p(d\lambda)+ \frac{\beta\, \mathcal K(p,\pi)}{n+1}\\
&&+\ \ \frac{16BL^2 (n+1)(2\log n)^{2/\kappa}}{n^2\beta\, t^{2/\kappa}}\,
, \quad \forall \ p\in \scrP_\Lambda.\nonumber
\end{eqnarray}
In particular, if $\Lambda=\{1,\dots,M\}$ and $\pi$ is the uniform
measure on $\Lambda$ we get
\begin{eqnarray}\label{oracle4}
E\Big(\|\hat f_n-f\|_n^2\Big)&\le& \min_{j=1,\dots,M} \|f_j-f\|_n^2
+ \frac{\beta\, \log M}{n+1}\\
&&+\ \ \frac{16BL^2 (n+1)(2\log n)^{2/\kappa}}{n^2\beta\,
t^{2/\kappa}}\,.\nonumber
\end{eqnarray}
\end{corollary}
\begin{proof}
Set $\alpha=(2(\log n)/t)^{2/\kappa}$ and note that
\begin{eqnarray}\label{oracle4_1}
R_n&\le& \frac{4(n+1)}{n\beta}\sup_{\lambda\in\Lambda,
\mu\in\scrP_\Lambda'} \|f_\lambda-\f_\mu\|_n^2\ \sum_{i=1}^n
E(\xi_i^2-\alpha)_+
\\
&\le& \frac{16L^2(n+1)}{\beta} E(\xi_1^2-\alpha)_+ \nonumber
\end{eqnarray}
where $a_+=\max(0,a)$. For any $x\geq (2/(t\kappa))^{1/\kappa}$ the
function $x\mapsto x^2e^{-tx^\kappa}$ is decreasing. Therefore, for
any $n\geq e^{1/\kappa}$ we have $x^2e^{-tx^\kappa}\leq \alpha
e^{-t\alpha^{\kappa/2}}=\alpha/n^2$, as soon as $x^2\ge \alpha$.
Hence, $E(\xi_1^2-\alpha)_+\le B\alpha/n^2$ and the desired
inequality follows.
\end{proof}
\begin{corollary}\label{cor1_2} Assume that $\sup_{\lambda\in\Lambda}\|f-f_\lambda\|_\infty\le L$ and the errors $\xi_i$ are symmetric with $E(|\xi_i|^s)\le B$
for some constants $s\ge 2$, $B>0$. Then for any $\alpha_0>0$ and
any $\beta\geq 4(1+1/n)\alpha_0 n^{2/(s+2)}+2L^2$ we have
\begin{eqnarray}\label{oracle3_2}
E\Big(\|\hat f_n-f\|_n^2\Big)&\le& \int_\Lambda
\|f_\lambda-f\|_n^2\, p(d\lambda)+ \frac{\beta\, \mathcal
K(p,\pi)}{n+1} + {\bar C} n^{- s/(s+2)} \, , \quad \forall \
p\in \scrP_\Lambda.\nonumber
\end{eqnarray}
where ${\bar C}>0$ is a constant that depends only on $s,L,B$ and
$\alpha_0$.
\end{corollary}
\begin{proof} Set $\alpha = \alpha_0 n^{2/(s+2)}$. In view of the inequality
$(f_\lambda(x_i)-\bar f_{\theta\cdot \pi}(x_i))^2\le 4\sup_{\lambda\in\Lambda}\|f-f_\lambda\|_\infty^2$,
the remainder term of Theorem~\ref{thm3} can be bounded as follows:
$$
R_n\le \frac{16L^2(n+1)}{n^2\beta}\sum_{i=1}^n E(\xi_i^2-\alpha)_+\le \frac{4L^2}{\alpha}\ E(\xi_1^2-\alpha)_+.
$$
To complete the proof, it suffices to notice that
$E(\xi_1^2-\alpha)_+=E(\xi_1^2 \1(\xi_1^2>\alpha)) \le
E(|\xi_1|^s)/\alpha^{s/2-1}$ by the Markov inequality.
\end{proof}

Corollary~\ref{cor1} shows that if the tails of the distribution of
errors have exponential decay and if $\beta$ is of the order $(\log
n)^{2/\kappa}$, then the rate of convergence in the bound
(\ref{oracle4}) is of the order $(\log n)^{\frac2\kappa}(\log M)/n$.
The residual $R_n$ in Corollary~\ref{cor1} is of a smaller order
than this rate and can be made even further smaller by taking
$\alpha=(u(\log n)/t)^{2/\kappa}$ with $u> 2$. For $\kappa=1$,
comparing Corollary~\ref{cor1} with the risk bounds obtained in
\cite{cat99,jrt05} for an averaged algorithm in i.i.d. random design
regression, we see that an extra $\log n$ multiplier appears. It is
noteworthy that this deterioration of the convergence rate does not
occur if only the existence of finite (power) moments is
assumed. In this case, the result of Corollary~\ref{cor1_2} provides
the same rates of convergence as those obtained under the analogous
moment conditions for model selection type aggregation in the i.i.d.
case, cf.\  \cite{jrt05,a2}.

\section{Sparsity oracle inequalities with no assumption on the dictionary}
\label{sec7}

In this section we assume that $f_\lambda$ is a
linear combination of $M$ known functions $\phi_1,\ldots,\phi_M$,
where $\phi_j:\cal X\to \RR$, with the vector of weights
$\lambda=(\lambda_1,\dots,\lambda_M)$ that belongs to a subset
$\Lambda$ of $\RR^M $:
$$f_\lambda=\sum_{j=1}^M \lambda_j\phi_j.$$
The set of functions $\{\phi_1,\ldots,\phi_M\}$ is called the
dictionary.

Our aim is to obtain sparsity oracle inequalities (SOI) for the
aggregate with exponential weights $\hat f_n$. The SOI are oracle
inequalities bounding the risk in terms of the number
$M(\lambda)$ of non-zero components (sparsity index) of $\lambda$
or similar characteristics. As discussed in Introduction, the
SOI is a powerful tool allowing one to solve simultaneously
several problems: sparse recovery in high-dimensional regression
models, adaptive nonparametric regression estimation, linear,
convex and model selection type aggregation.

For $\lambda\in\RR^M$ denote by $J(\lambda)$ the set~of indices $j$
such that $\lambda_j\ne0$, and set $M(\lambda)\triangleq {\rm
Card}(J(\lambda))$. For any $\tau>0$, $0<L_0\le\infty$, define the
probability densities
\begin{align}\label{prior1}
q_0(t)&=\frac3{2(1+|t|)^4},\ \forall t\in\RR,\\
q(\lambda)&=\frac1{C_0} \prod_{j=1}^M \tau^{-1}\;
q_0\big(\lambda_j/\tau\big)\1(\|\lambda\|\le L_0),\
\forall\lambda\in\RR^M,\label{prior2}
\end{align}
where $C_0=C_0(\tau,M,L_0)$ is a normalizing constant such that $q$
integrates to 1, and $\|\lambda\|$ stands for the Euclidean norm of
$\lambda\in\RR^M$.

In this section we choose the prior $\pi$ in the definition of
$f_\lambda$ as a distribution on~$\RR^M$ with the Lebesgue density
$q$:  $\pi(d\lambda)=q(\lambda)d\lambda.$ We will call it the {\it
sparsity prior}.

 Let us now discuss this choice of the prior. Assume for simplicity that
 $L_0=\infty$ which implies $C_0=1$.
 Then the aggregate $\hat f_n$ based on the sparsity prior
 can be written in the form
$\hat f_{n}=f_{\hat\lambda}$, where $\hat\lambda=(\hat
\lambda_1,\dots,\hat \lambda_M)$ is the posterior mean in the
``phantom" parametric model (\ref{phm}):
$$
\hat \lambda_j = \int_{\RR^M} \lambda_j \theta_n(\lambda) d\lambda,
\quad j=1,\dots, M,
$$
with the posterior density
\begin{eqnarray}\label{phm1}
\theta_n(\lambda) &=& C \exp\big\{-n\|{\bf Y}-{
f}_\lambda\|_n^2/\beta + \log q(\lambda)\big\}\\
&=&C'\exp\big\{-n\|{\bf Y}-{ f}_\lambda\|_n^2/\beta -
4\sum_{j=1}^M\log (1+|\lambda_j|/\tau)\big\}\,.\nonumber
\end{eqnarray}
Here $C>0, C'>0$ are normalizing constants, such that
$\theta_n(\cdot)$ integrates to 1. To compare our estimator with
those based on the penalized least squares approach (BIC, Lasso,
bridge), we consider now the posterior mode $\tilde \lambda$ of
$\theta_n(\cdot)$ (the MAP estimator) instead of the posterior mean
$\hat \lambda$. It is easy to see that $\tilde \lambda$ is also a
penalized least squares estimator. In fact, it follows from
(\ref{phm1}) that the MAP estimator is a solution of the
minimization problem
\begin{eqnarray}\label{phm2} \tilde \lambda =
\arg \min_{\lambda\in\RR^M} \Big\{ \|{\bf Y} -{ f}_\lambda\|_n^2 +
\frac{4\beta}{n}\sum_{j=1}^M\log (1+|\lambda_j|/\tau)\Big\}.
\end{eqnarray}
Thus, the MAP ``approximation" of our estimator suggests that it can
be heuristically associated with the penalty which is logarithmic in
$\lambda_j$. In the sequel, we will choose $\tau$ very small (cf.\
Theorems 5 and 6 below). For such values of $\tau$ the function
$\lambda_j \mapsto \log (1+|\lambda_j|/\tau)$ is very steep near the
origin and can be viewed as a reasonable approximation for the BIC
penalty function $\lambda_j \mapsto \1(\lambda_j\ne 0)$. The penalty
$\log (1+|\lambda_j|/\tau)$ is not convex in $\lambda_j$, so that
the computation of the MAP estimator (\ref{phm2}) is problematic,
similarly to that of the BIC estimator. On the other hand, our
posterior mean $\hat f_{n}$ is efficiently computable. Thus, the
aggregate $\hat f_{n}$ with the sparsity prior can be viewed as a
computationally feasible approximation to the logarithmically
penalized least squares estimator or to the closely related BIC
estimator. Interestingly, the results that we obtain below for the
estimator $\hat f_n$ are valid under weaker conditions than the
analogous results for the Lasso and Dantzig selector proved in
\cite{brt07,btw07b} and are sharper than those for the BIC
\cite{btw07a} since we get oracle inequalities with leading constant
1 that are not available for the BIC.

Note that if we redefine $q_0$ as the double exponential density,
the corresponding MAP estimator is nothing but the penalized least
squares estimator with the Lasso penalty $\sim \sum_{j=1}^M
|\lambda_j|$. More generally, if $q_0(t)\sim \exp(-|t|^{\gamma})$
for some $0<\gamma<2$, the corresponding MAP solution is a bridge
regression estimator, i.e., the penalized least squares estimator
with penalty $\sim \sum_{j=1}^M |\lambda_j|^\gamma$ \cite{ff93}. The
argument that we develop below can be easily adapted for these
priors, but the resulting SOI are not as accurate as those that we
obtain in Theorems 5 and 6 for the sparsity prior (\ref{prior1}),
(\ref{prior2}). The reason is that the remainder term of the SOI is
logarithmic in $\lambda_j$ when the sparsity prior is used, whereas
it increases polynomially in $\lambda_j$ for the above mentioned
priors.

We first prove a theorem that provides a general tool to derive the
SOI from the PAC-Bayesian bound (\ref{oracle6}). Then we will use it
to get the SOI in more particular contexts. Note that in this
general theorem $\hat f_n$ is not necessarily an exponentially
weighted aggregate defined by (\ref{aggr}). It can be any $\hat f_n$
satisfying (\ref{oracle6}). The result of the theorem obviously
extends to the case where a remainder term as $R_n$ (cf.\
(\ref{oracle1})) is added to the basic PAC-Bayesian bound
(\ref{oracle6}).

\begin{theorem}\label{thm4}
Let $\hat f_n$ satisfy (\ref{oracle6}) with
$\pi(d\lambda)=q(\lambda)\,d\lambda$ and $\tau\le \delta L_0/\sqrt
M$ where $0<L_0\le \infty$, $0<\delta <1$. Assume that $\Lambda$
contains the ball $\{\lambda\in\RR^M: \|\lambda\|\le L_0\}$. Then
for all $\lambda^*$ such that $\|\lambda^*\|\le (1-\delta)L_0$ we
have
$$
E\Big(\|\hat f_n-f\|_n^2\Big)\le \|f_{\lambda^*}-f\|_n^2+
\frac{4\beta}{n}\sum_{j\in
J(\lambda^*)}\log(1+\tau^{-1}|\lambda_j^*|)+R(M,\tau,L_0,\delta),
$$
where the residual term is
$$
R(M,\tau,L_0,\delta)=\tau^2e^{2\tau^3M^{5/2}(\delta
L_0)^{-3}}\sum_{j=1}^M\|\phi_j\|_n^2+
\frac{2\beta\tau^3M^{5/2}}{n\delta^3L^3_0}
$$
for $L_0<\infty$ and $R(M,\tau,\infty,\delta)=\tau^2\sum_{j=1}^M\|\phi_j\|_n^2$.
\end{theorem}
\begin{proof}
We apply Theorem~\ref{thm1} with
$p(d\lambda)=C_{\lambda^*}^{-1}q(\lambda-\lambda^*)\1(\|\lambda-\lambda^*\|\le \delta L_0)\,d\lambda$, where
$C_{\lambda^*}$ is the normalizing constant. Using the symmetry of
$q$ and the fact that
$f_{\lambda}-f_{\lambda^*}=f_{\lambda-\lambda^*}=-f_{\lambda^*-\lambda}$
we get
$$
\int_\Lambda \langle
f_{\lambda^*}-f,f_\lambda-f_{\lambda^*}\rangle_n\,p(d\lambda)=C_{\lambda^*}^{-1}
\int_{\|w\|\le \delta L_0}
\langle f_{\lambda^*}-f,f_w\rangle_n\,q(w)\,dw =0.$$ Therefore
$\int_\Lambda\|f_\lambda-f\|_n^2\,p(d\lambda)=
\|f_{\lambda^*}-f\|_n^2+\int_\Lambda\|f_\lambda-f_{\lambda^*}\|_n^2\,p(d\lambda)$. On the other hand,
bounding the indicator  $\1(\|\lambda-\lambda^*\|\le \delta L_0)$ by one and using the identities
$\int_\RR q_0(t)\,dt=\int_\RR t^2q_0(t)\,dt=1$, we obtain
\begin{align*}
\int_\Lambda\|f_\lambda-f_{\lambda^*}\|_n^2\,p(d\lambda)&\le
%&= C_{\lambda^*}^{-1}\int_{\|w\|\le\delta L_0}\|f_w\|_n^2\,q(w)\,dw\\
\frac1{C_0C_{\lambda^*}}\sum_{j=1}^M\|\phi_j\|_n^2\int_{\RR}\frac{w_j^2}{\tau}\;q_0\Big(
\frac{w_j}{\tau}\Big)\,dw_j
=\frac{\tau^2\sum_{j=1}^M\|\phi_j\|_n^2}{C_0C_{\lambda^*}} .
\end{align*}
Since $1-x\ge e^{-2x}$ for all $x\in[0,1/2]$, we get
\begin{align*}
C_{\lambda^*}C_0&=\frac1{\tau^M}\int_{\|\lambda\|\le \delta L_0}
\Big\{ \prod_{j=1}^M q_0\Big(\frac{\lambda_j}{\tau}\Big)\Big\}
\,d\lambda\ge\frac1{\tau^M}\prod_{j=1}^M \Big\{\int_{|\lambda_j|\le
\frac{\delta L_0}{\sqrt M}}
q_0\Big(\frac{\lambda_j}{\tau}\Big)\,d\lambda_j\Big\}\\
&=\bigg(\int_0^{\delta L_0/\tau\sqrt M} \frac{3dt}{(1+t)^4}\bigg)^M
=\bigg(1-\frac{1}{(1+\delta L_0\tau^{-1} M^{-1/2})^3}\bigg)^M\\
&\ge \exp\Big(-\frac{2M}{(1+\delta L_0\tau^{-1} M^{-1/2})^3}\Big)
\ge \exp(-2\tau^3M^{5/2}(\delta L_0)^{-3}).
\end{align*}
On the other hand, in view of the inequality $1+|\lambda_j/\tau|\le (1+|\lambda^*_j/\tau|)(1+|\lambda_j
-\lambda_j^*|/\tau)$ the Kullback-Leibler divergence between $p$ and
$\pi$ is bounded as follows:
\begin{align*}
{\cal K} (p,\pi)&=\int_{\RR^M}\log\bigg(
\frac{C_{\lambda^*}^{-1}q(\lambda-\lambda^*)}{q(\lambda)}\bigg)\,p(d\lambda)
\le
 4\sum_{j=1}^M\log(1+|\tau^{-1}\lambda_j^*|)-\log C_{\lambda^*}\,.
\end{align*}
Easy computation yields $C_0\le 1$. Therefore
$C_{\lambda^*}\ge C_0C_{\lambda^*}\ge \exp(-\frac{2\tau^3M^{5/2}}{(\delta L_0)^{3}})$ and the
desired result follows.
\end{proof}

Inspection of the proof of Theorem \ref{thm4} shows that our choice
of prior density $q_0$ in (\ref{prior1}) is not the only possible
one. Similar result can be readily obtained when $q_0(t)\sim
|t|^{-3-\delta}$, as $|t|\to\infty$, for any $\delta>0$. The
important point is that $q_0(t)$ should be symmetric, with finite
second moment, and should decrease not faster than a polynomial, as
$|t|\to\infty$.

We now explain how the result of Theorem \ref{thm4}
can be applied to improve the SOI existing in the literature. In our
setup the values $x_1,\dots, x_n$ are deterministic. For this case,
SOI for the BIC, Lasso and Dantzig selector are obtained in
\cite{btw07a,ct,zh07,brt07}. In those papers the random errors
$\xi_i$ are Gaussian. So, we will also focus on the Gaussian case,
though similar corollaries of Theorem 5 are straightforward to
obtain for other distributions of errors satisfying the assumptions
of Sections 3, 4 or 6.

Denote by $\Phi$ the Gram matrix associated
to the family $(\phi_j)_{j=1,\ldots,M}$, i.e., the $M\times M$
matrix with entries $\Phi_{j,j'}=n^{-1}\sum_{i=1}^n
\phi_j(x_i)\phi_{j'}(x_i)$, $j,j'\in\{1,\ldots,M\}$, and denote by
$\tr(\Phi)$ the trace of $\Phi$. Set $\log_+x=\max(\log x,0)$,
$\forall \ x>0$.

\begin{theorem}\label{soi}
Let $\hat f_n$ be defined by (\ref{aggr}) with
$\pi(d\lambda)=q(\lambda)\,d\lambda$ and $L_0=\infty$.  Let $\xi_i$
be i.i.d.\ Gaussian $\mathcal N(0,\sigma^2)$ random variables with
$\sigma^2>0$ and assume that $\beta\ge 4\sigma^2$, $\tr(\Phi)>0$.
Set $\tau=\frac{\sigma}{\sqrt{n\tr(\Phi)}}$. Then for all
$\lambda^*\in \RR^M$ we have
\begin{align*}
E\Big(\|\hat f_n-f\|^2_n\Big)&\le \|f_{\lambda^*}-f\|_n^2 +
\frac{4\beta M(\lambda^*)}{n}
\Big(1\!+\log_+\!\!\Big\{\frac{\sqrt{n\tr(\Phi)}}{M(\lambda^*)\sigma}
|\lambda^*|_1\Big\}\!\Big)+\frac{\sigma^2}{n}
\end{align*}
where $|\lambda^*|_1=\sum_{j=1}^M|\lambda_j^*|$.
\end{theorem}

\begin{proof}
To apply Theorem~\ref{thm4} with $L_0=\infty$, we need to verify
that $\hat f_n$ satisfies (\ref{oracle6}). This is indeed the case
in view of Theorem~\ref{thm6}. Thus we have
\begin{align}\label{ssoi}
E\Big(\|\hat f_n-f\|_n^2\Big)\le \|f_{\lambda^*}-f\|_n^2+
\frac{4\beta}{n}\sum_{j\in
J(\lambda^*)}\log(1+\tau^{-1}|\lambda_j^*|)+\tau^2\tr(\Phi).
\end{align}
By Jensen's inequality, $\sum_{j\in
J(\lambda^*)}\log(1+\tau^{-1}|\lambda_j^*|)\le M(\lambda^*)
\log(1+|\lambda^*|_1/(\tau M(\lambda^*)))$.  Since
$\log(1+|\lambda^*|_1/(\tau M(\lambda^*)))\le
1+\log_+(|\lambda^*|_1/(\tau M(\lambda^*)))$, the result of the
theorem follows from the choice of $\tau$.
\end{proof}

Theorem~\ref{soi} establishes a SOI {\it with leading
constant 1} and {\it with no assumption on the dictionary}. Of
course, for the inequality to be meaningful, we need a mild
condition on the dictionary: ${\rm Tr} (\Phi)<\infty$. But this is
even weaker than the standard normalization assumption
$\|\phi_j\|_n^2=1$, $j=1,\dots,M$. Note that a BIC type aggregate
also satisfies a SOI similar to that of Theorem~\ref{soi} with no
assumption on the dictionary (cf.\  \cite{btw07a}), but with leading
constant greater than 1. However, it is well-known that the
BIC is not computationally feasible, unless the dimension $M$ is
very small (say, $M=20$ in the uppermost case), whereas our
estimator can be efficiently computed for much larger $M$.

The oracle inequality of Theorem~\ref{soi} can be compared with the
analogous SOI obtained for the Lasso and Dantzig selector under
deterministic design \cite{btw07a,brt07}. Similar oracle
inequalities for the case of random design $x_1,\dots,x_n$ can be
found in \cite{btw07b,vdg06,k06}. All those results impose heavy
restrictions on the dictionary in terms of the coherence introduced
in \cite{det06} or other analogous characteristics that limit the
applicability of the corresponding SOI, see the discussion after
Corollary~\ref{soi1} below.

We now turn to the problem of high-dimensional parametric linear
regression, i.e., to the particular case of our setting when there
exists $\lambda^*\in \RR^M$ such that $f=f_{\lambda^*}$. This is the
framework considered in \cite{ct,zh07} and also covered as an
example in \cite{brt07}. In these papers it was assumed that the
basis functions are normalized: $\|\phi_j\|_n^2=1$, $j=1,\dots,M$,
and that some restrictive assumptions on the eigenvalues of the
matrix $\Phi$ hold. We only impose a very mild condition:
$\|\phi_j\|_n^2\le \phi_0$, $j=1,\dots,M,$ for some constant
$\phi_0<\infty$.

\begin{corollary}\label{soi1}
Let $\hat f_n$ be defined by (\ref{aggr}) with
$\pi(d\lambda)=q(\lambda)\,d\lambda$ and $L_0=\infty$.  Let $\xi_i$
be i.i.d.\ Gaussian $\mathcal N(0,\sigma^2)$ random variables with
$\sigma^2>0$ and assume that $\beta\ge 4\sigma^2$. Set
$\tau=\frac{\sigma}{\sqrt{\phi_0nM}}$. If there exists $\lambda^*\in
\RR^M$ such that $f=f_{\lambda^*}$ and $\|\phi_j\|_n^2\le \phi_0$,
$j=1,\dots,M,$ for some $\phi_0<\infty$, we have
\begin{eqnarray}\label{soip}
&&E\Big(\|\hat f_n-f\|^2_n\Big)\le \frac{4\beta}{n} M(\lambda^*)
\Big(1\!+\log_+\!\!\Big\{\frac{\sqrt{\phi_0 nM}}{M(\lambda^*)\sigma}
|\lambda^*|_1\Big\}\!\Big)+\frac{\sigma^2}{n}\,.
\end{eqnarray}
\end{corollary}

Proof is based on the fact that ${\rm Tr} (\Phi)=
\sum_{j=1}^M\|\phi_j\|_n^2 \le M\phi_0$ in (\ref{ssoi}).

Under the assumptions of Corollary \ref{soi1}, the rate of
convergence of $\hat f_n$ is of the order $O(M(\lambda^*)/n)$, up to
a logarithmic factor. This illustrates the sparsity property of the
exponentially weighted aggregate $\hat f_n$: if the (unknown) number
of non-zero components $M(\lambda^*)$ of the true parameter vector
$\lambda^*$ is much smaller than the sample size $n$, the estimator
$\hat f_n$ is close to the regression function $f$, even when the
nominal dimension $M$ of $\lambda^*$ is much larger than $n$. In
other words, $\hat f_n$ achieves approximately the same performance
as the ``oracle" ordinary least squares that knows the set
$J(\lambda^*)$ of non-zero components of $\lambda^*$. Note that
similar performance is proved for the Lasso and Dantzig selector
\cite{btw07a,ct,zh07,brt07}, however the risk bounds analogous to
(\ref{soip}) for these methods are of the form
$O\big(M(\lambda^*)(\log M)/(\kappa_{n,M}n)\big)$, where
$\kappa_{n,M}$ is a ``restricted eigenvalue" of the matrix $\Phi$
which is assumed to be positive (see \cite{brt07} for a detailed
account). This kind of assumption is violated for many important
dictionaries, such as the decision stumps, cf.\  \cite{brt07}, and
when it is satisfied the eigenvalues $\kappa_{n,M}$ can be rather
small. This indicates that the bounds for the Lasso and Dantzig
selector can be quite inaccurate as compared to (\ref{soip}).

\section{Appendix}

\begin{lemma}
\label{lem2} For any $x\in\RR$ and any $\alpha_0>0$, $
x+\log\big(1+\frac1{\alpha_0}\big(e^{-x\alpha_0}-1\big)\big)\leq
\frac{x^2\alpha_0}2.$
\end{lemma}
\begin{proof}
On the interval $(-\infty,0]$, the function $x\mapsto
x+\log\big(1+\frac1{\alpha_0}(e^{-x\alpha_0}-1)\big)$ is increasing,
therefore it is bounded by its value at $0$, that is by $0$. For
positive values of $x$, we combine the inequalities $e^{-y}\leq
1-y+y^2/2$ (with $y=x\alpha_0$) and $\log(1+y)\leq y$ (with
$y=1+\frac1{\alpha_0}(e^{-x\alpha_0}-1)$).
\end{proof}

\begin{lemma}\label{lem3}
For any $\beta\geq s^2/n+2\sup_{\lambda\in\Lambda}\|f-f_\lambda
\|^2_n$ and for every $\mu'\in\scrP_\Lambda'$, the function
$$
\mu\mapsto \exp\Big(\frac{s^2\|\f_{\mu'} -\f_{\mu}\|_n^2}{n\beta^2}-
\frac{\|f-\f_\mu\|_n^2}\beta\Big)
$$
is concave.
\end{lemma}
\begin{proof} Consider first the case where ${\rm Card}(\Lambda)=m<\infty$.
Then every element of $\scrP_\Lambda$ can be viewed as a vector from
$\RR^m$. Set
\begin{align*}
Q(\mu)&=(1-\gamma)\|f-f_\mu\|_n^2+2\gamma\langle f-f_\mu,
f-f_{\mu'}\rangle_n\\
&=(1-\gamma)\mu^TH_n^TH_n\mu+2\gamma \mu^TH_n^TH_n\mu',
\end{align*}
where $\gamma=s^2/(n\beta)$ and $H_n$ is the $n\times m$ matrix with
entries $(f(x_i)-f_\lambda (x_i))/\sqrt n$. The statement of the
lemma is equivalent to the concavity of $e^{-Q(\mu)/\beta}$ as a
function of $\mu\in\scrP_\Lambda$, which holds if and only if the
matrix $\beta\nabla^2 Q(\mu)-\nabla Q(\mu)\nabla Q(\mu)^T$ is
positive-semidefinite. Simple algebra shows that $\nabla^2
Q(\mu)=2(1-\gamma)H_n^TH_n$ and $\nabla Q(\mu)=
2H_n^T[(1-\gamma)H_n\mu+\gamma H_n\mu']$. Therefore, $\nabla Q(\mu)
\nabla Q(\mu)^T=H_n^T\bM H_n$, where
$\bM=4H_n\tilde\mu\tilde\mu^TH_n^T$ with
$\tilde\mu=(1-\gamma)\mu+\gamma\mu'$. Under our assumptions, $\beta$
is larger than $s^2/n$, ensuring thus that
$\tilde\mu\in\scrP_\Lambda$. Clearly, $\bM$ is a symmetric and
positive-semidefinite matrix. Moreover,
\begin{align*}
\lambda_{max}(\bM)&\leq \tr(\bM)=4\|H_n\tilde\mu\|^2=\frac{4}{n}
\sum_{i=1}^n\bigg(\sum_{\lambda\in\Lambda} \tilde\mu_\lambda
(f-f_\lambda )(x_i)\bigg)^2\\
&\le \frac{4}{n}\sum_{i=1}^n \sum_{\lambda\in\Lambda}
\tilde\mu_\lambda (f(x_i)-f_\lambda (x_i))^2
=4\sum_{\lambda\in\Lambda}\tilde\mu_\lambda \|f-f_\lambda \|^2_n\\
&\le 4 \max_{\lambda\in\Lambda}\|f-f_\lambda \|_n^2
\end{align*}
where $\lambda_{max}(\bM)$ is the largest eigenvalue of $\bM$ and $\tr(\bM)$ is its
trace. This estimate yields the matrix inequality
$$
\nabla Q(\mu)\nabla Q(\mu)^T\le 4
\max_{\lambda\in\Lambda}\|f-f_\lambda \|_n^2\,H_n^TH_n.
$$
Hence, the function $e^{-Q(\mu)/\beta}$ is concave as soon as
$4\max_{\lambda\in\Lambda}\|f-f_\lambda \|_n^2\le 2\beta(1-\gamma)$.
The last inequality holds for every $\beta\geq
n^{-1}s^2+2\max_{\lambda\in\Lambda}\|f-f_\lambda \|^2_n$.

The general case can be reduced to the case of finite $\Lambda$ as
follows. The concavity of the functional
$G(\mu)=\exp\Big(\frac{s^2\|\f_{\mu'}
-\f_{\mu}\|_n^2}{n\beta^2}-\frac{\|f-\f_\mu\|_n^2}\beta\Big)$ is
equivalent to the validity of the inequality
\begin{align}\label{G}
G\Big(\frac{\mu+\tilde\mu}2\Big)\ge
\frac{G(\mu)+G(\tilde\mu)}{2},\qquad \forall \ \mu,\tilde\mu \in\pp.
\end{align}
Fix now arbitrary $\mu,\tilde\mu \in\pp$. Take
$\tilde\Lambda=\{1,2,3\}$ and consider the set of functions
$\{\tilde f_\lambda,\lambda\in\tilde\Lambda\}=
\{\f_\mu,\f_{\tilde\mu},\f_{\mu'}\}$. Since $\tilde\Lambda$ is
finite, $\scrP_{\tilde\Lambda}'=\scrP_{\tilde\Lambda}$. According to
the first part of the proof, the functional
$$\tilde
G(\nu)=\exp\bigg(\frac{s^2\|\f_{\mu'} -\bar{\tilde
f}_{\nu}\|_n^2}{n\beta^2}-\frac{\|f-\bar{\tilde
f}_\nu\|_n^2}\beta\bigg), \quad \nu\in \scrP_{\tilde\Lambda},
$$
is concave on $\scrP_{\tilde\Lambda}$ as soon as $\beta\geq
s^2/n+2\max_{\lambda\in\tilde\Lambda}\|f-\tilde f_\lambda \|^2_n$,
and therefore for every $\beta\geq
s^2/n+2\sup_{\lambda\in\Lambda}\|f-f_\lambda \|^2_n$ as well.
(Indeed, by Jensen's inequality for any measure $\mu\in\pp$ we have
$\|f-\f_\mu \|^2_n \le \int \|f-f_\lambda \|^2_n \mu(d\lambda)\le
\sup_{\lambda\in\Lambda}\|f-f_\lambda \|^2_n$.) This leads to
$$
\tilde G\Big(\frac{\nu+\tilde\nu}2\Big)\ge \frac{\tilde
G(\nu)+\tilde G(\tilde\nu)}{2},\qquad \forall \ \nu,\tilde\nu
\in\scrP_{\tilde\Lambda}.
$$
Taking here the Dirac measures $\nu$ and $\tilde\nu$ defined by
$\nu(\lambda=j)=\1(j=1)$ and $\tilde\nu(\lambda=j)=\1(j=2)$,
$j=1,2,3$, we arrive at (\ref{G}). This completes the proof of the
lemma.
\end{proof}

\end{document}